\documentclass[12pt,leqno]{amsart}
\usepackage{amssymb,amsthm,amsmath,latexsym, color}
\usepackage[all]{xy}
\newcommand{\vfi}{\varphi}

\newtheorem{theorem}{\sc Theorem}[section]

\newtheorem{lem}[theorem]{\sc Lemma}
\newtheorem{prop}[theorem]{\sc Proposition}

\newtheorem{rem}[theorem]{\sc Remark}
\newtheorem{ex}[theorem]{\sc Example}
\newtheorem{hyp}[theorem]{\sc Hypothesis}

\newtheorem*{thmA}{Theorem A}
\newtheorem*{corB}{Corollary B}

\title[Non-abelian tensor square]{Boundedly finite Conjugacy classes of tensors}
\author[Bastos]{ Raimundo Bastos}
\address{ Departamento de Matem\'atica, Universidade de Bras\'ilia,
Brasilia-DF, 70910-900 Brazil }
\email{bastos@mat.unb.br}
\author[Monetta]{ Carmine Monetta }
\address{Dipartimento di Matematica, Universit\`a di Salerno, Via Giovanni Paolo II, 132 - 84084 - Fisciano (SA), Italy}
\email{cmonetta@unisa.it}

\begin{document}

\maketitle

\begin{abstract}
Let $n$ be a positive integer and let $G$ be a group. We denote by $\nu(G)$ a certain extension of the non-abelian tensor square $G \otimes G$ by $G \times G$. Set $T_{\otimes}(G) = \{g \otimes h \mid g,h \in G\}$. We prove that if the size of the conjugacy class $\left |x^{\nu(G)} \right| \leq n$ for every $x \in T_{\otimes}(G)$, then the second derived subgroup $\nu(G)''$ is finite with $n$-bounded order. Moreover, we obtain a sufficient condition for a group to be a BFC-group.
\end{abstract}

\vspace{0,5cm}

\thanks{\it Dedicated to Pavel Shumyatsky on the occasion of his 60th birthday}

\maketitle

\section{Introduction}

The non-abelian tensor square $G \otimes G$ of a group $G$ as introduced by R. Brown and J.\,L. Loday \cite{BL} is defined to be the group generated by all symbols $\; \, g\otimes h, \; g,h\in G$, subject to the relations
\[
gg_1 \otimes h = ( g^{g_1}\otimes h^{g_1}) (g_1\otimes h) \quad
\mbox{and} \quad g\otimes hh_1 = (g\otimes h_1)( g^{h_1} \otimes
h^{h_1})
\]
for all $g,g_1, h,h_1 \in G$. In the same paper, R. Brown and J.-L. Loday presented a topological significance for the non-abelian tensor square of groups (cf. \cite[Section 3]{BL}). The study of the non-abelian tensor square of groups from a group theoretic point of view was initiated by R. Brown, D.\,L. Johnson and E.\,F. Robertson \cite{BJR}. 
If $x$ and $y$ are group-elements, we denote by $x^y=y^{-1}xy$ the conjugate of $x$ by $y$, while the commutator of $x$ and $y$ is the element $[x,y]=x^{-1}x^y$. 
We observe that the defining relations of the tensor square can be viewed as abstractions of commutator relations; thus in \cite{NR1} the following construction is considered. Let $G$ be a group and $\varphi : G \rightarrow G^{\varphi}$ an isomorphism ($G^{\varphi}$ is an isomorphic copy of $G$, where $g \mapsto g^{\varphi}$, for all $g \in G$). Define the group $\nu(G)$ to be \[ \nu (G):= \langle 
G \cup G^{\varphi} \ \vert \ [g_1,{g_2}^{\varphi}]^{g_3}=[{g_1}^{g_3},({g_2}^{g_3})^{\varphi}]=[g_1,{g_2}^{\varphi}]^{{g_3}^{\varphi}},
\; \ g_i \in G \rangle .\]
The motivation for studying $\nu(G)$ is the commutator connection: indeed, the map  $\Phi: G \otimes G \rightarrow [G, G^{\varphi}]$,
defined by $g \otimes h \mapsto [g , h^{\varphi}]$, for all $g, h \in G$, is an isomorphism  (N.\,R. Rocco, \cite[Proposition 2.6]{NR1}). Therefore, from now on we identify the non-abelian tensor square $G \otimes G$ with the subgroup $[G,G^{\vfi}]$ of $\nu(G)$. An element $\alpha \in \nu(G)$ is called a {\em tensor} if $\alpha = [a,b^{\varphi}]$ for suitable $a,b\in G$. We write $T_{\otimes}(G)$ to denote the set of all tensors (in $\nu(G)$). In particular, 
$T_{\otimes}(G)$ is a commutator closed subset of the group $\nu(G)$, wherea subset $X$ of a group is called commutator closed if $[x,y]\in X$ for any $x,y\in X$. 

It is well-known that the set of all tensors $T_{\otimes}(G) $ affects the structure of the non-abelian tensor square $G \otimes G$, and of related constructions (see \cite{BNR, BNR2}). For instance, in \cite{BNR} and in \cite{BNR2}, it was proved that if the set  $T_{\otimes}(G)$ is finite (resp. has exactly $n$ elements), then the non-abelian tensor square $[G,G^{\varphi}]$ is finite (resp. finite with $n$-bounded order). Here and throughout the article we use the expression $\{a,b,\ldots\}$-bounded to mean that a quantity is bounded by a certain number depending only on the parameters $a, b,\ldots$.

Recall that a group $G$ is a BFC-group if every conjugacy class of $G$ contains at most $n$ elements, for a positive integer $n$. In \cite{neu}, B.\,H. Neumann proved that the group $G$ is a BFC-group if and only if the derived subgroup $G'$ is finite, and this occurs if and only if $G$ contains only finitely many commutators. Later, J. Wiegold proved a quantitative version of Neumann's result: if $G$ contains exactly $m$ commutators, then the order of the derived subgroup $G'$ is finite with $m$-bounded order \cite[Theorem 4.7]{W}, and the best known bound was obtained in \cite{gumaroti}. Recall that $x^G=\{x^g \ | \ g\in G\}$ denote the conjugacy class of $x$ in $G$. 
Recently, in \cite{DS}, G. Dierings and P. Shumyatsky proved that if $\left|x^G \right| \leq n$ for every commutator $x$ in $G$, then the second derived subgroup $G''$ is finite with $n$-bounded order. Subsequently, E. Detomi, M. Morigi and P. Shumyatsky obtain an interesting version for higher commutator subgroups (cf. \cite{DMS}).

In the present paper we consider non-abelian tensor square of groups (and related construction) in which tensors have a bounded number of conjugates, obtaining the following result.

\begin{thmA} \label{thm.main}
Let $n$ be a positive integer and $G$ a group. Suppose that the size of the conjugacy class $\left |x^{\nu(G)} \right| \leq n$ for every $x \in T_{\otimes}(G)$. Then the second derived subgroup $\nu(G)''$ is finite with $n$-bounded order. 
\end{thmA}

It is worth to mention that we manage to adapt the proof of \cite[Theorem 1.1]{DS} in the settings of tensors using the tools presented in~\cite{NR2}.

Furthermore, we also examine the structure of $G$ in terms of some finiteness properties of the set of tensors $T_{\otimes}(G)$, obtaining a sufficient condition for a group $G$ to be a BFC-group.

\begin{corB}\label{coro} 
Let $n$ be a positive integer. Let $G$ be a group in which the derived subgroup $G'$ is finitely generated. Assume that the size of the conjugacy class $|x^{\nu(G)}| \leq n$ for every $x \in T_{\otimes}(G)$. Then $G$ is a BFC-group.
\end{corB}

Finally, we show, by means of two counterexamples, that Corollary~B is no longer true neither if we discard the hypothesis about $G'$ to be finitely generated (see Example \ref{ex.prufer}, below) nor if we consider bounded conjugacy classes of commutators instead of tensors (see Example \ref{ex.dihedral}, below).   

\vspace{2mm}

The paper is organized as follows. In the next section we collect some preliminary results about the non-abelian tensor square, while Section 3 contains the proofs of our main results.

\section{The group $\nu(G)$}

It is well known that the finiteness of the non-abelian tensor square $[G,G^{\varphi}]$, does not imply that $G$ is a finite group. For instance, the Pr\"ufer group $C_{p^{\infty}}$ is an example of an infinite group such that the non-abelian tensor square $[C_{p^{\infty}},(C_{p^{\infty}})^{\varphi}]$ is trivial (and so, finite). This is the case for all torsion, divisible abelian groups. A useful result, due to Parvizi and Niroomand \cite{NP}, provides a sufficient condition for a group to be finite.

\begin{lem} \label{fg}
Let $G$ be a finitely generated group. Suppose that the non-abelian tensor square $[G,G^{\varphi}]$ is finite. Then $G$ is finite. 
\end{lem}

The following basic properties are consequences of 
the defining relations of $\nu(G)$ and commutator rules (see \cite[Section 2]{NR1} for more details). 

\begin{lem} 
\label{basic.nu}
The following relations hold in $\nu(G)$, for all 
$g, h, x, y \in G$.
\begin{itemize}
\item[(a)] $[g, h^{\varphi}]^{[x, y^{\varphi}]} = [g, h^{\varphi}]^{[x, 
y]}$; 
\item[(b)] $[g, h^{\varphi}, x^{\varphi}] = [g, h, x^{\varphi}] = [g, 
h^{\varphi}, x] = [g^{\vfi}, h, x^{\vfi}] = [g^{\vfi}, h^{\vfi}, x] = 
[g^{\vfi}, 
h, x]$;
\item[(c)] $[[g,h^{\varphi}],[x,y^{\varphi}]] = [[g,h],[x,y]^{\varphi}]$. 
\end{itemize}
\end{lem}

Given a group $G$ there exists an epimorphism \[\rho: \nu(G) \to G,\]  given by $g \mapsto g$, $h^{\vfi} \mapsto h$. In the notation of \cite[Section 2]{NR2}, let $\Theta(G)$ denote the kernel of $\rho$. The following lemma collects some results related to the subgroup $\Theta(G)$ which can be found in \cite[Section 2]{NR2}.   

\begin{lem} \ 
\label{lem.theta} \
\begin{itemize}
\item[(a)] The quotient group  $\nu(G)/\Theta(G)$ is isomorphic to $G$;
\item[(b)] The subgroup $\Theta(G)$ centralizes $[G,G^{\varphi}]$. 
\end{itemize}
\end{lem}

The next lemma is a crucial observation, that we will use several times in our proofs. 

\begin{lem} \ \label{lem.the}
Let $\alpha,\beta \in \nu(G)$. 
\begin{itemize}
    \item[(a)] $[x,y^{\varphi}]^{\alpha} = [x,y ^{\varphi}]^{\rho(\alpha)} = [x,y ^{\varphi}]^{\rho(\alpha)^{\varphi}}$, for any $x,y \in G$; 
    \item[(b)] There exists an element $\theta \in \Theta(G)$ such that $$[\alpha,\beta] = [u,v^{\varphi}] \theta,$$ where $u=\rho(\alpha)$ and $v = \rho(\beta)$. 
\end{itemize}

\end{lem}

\begin{proof}
(a) It holds by defining relations of $\nu(G)$.  \\

(b) Note that $$ \rho([\alpha,\beta]) = [u,v] =  \rho([u,v^{\varphi}]).$$  Since $\Theta(G)$ is the kernel  of $\rho$, it follows that there exists an element $\theta \in \Theta(G)$ such that $$ [\alpha,\beta] = [u,v^{\varphi}]\theta,$$ as well. The proof is complete. 
\end{proof}

Given an element $\alpha \in [G,G^{\varphi}]$, we define $l_{\otimes}(\alpha)$ to be the smallest positive integer such that $\alpha$ may be written as a product of $l_{\otimes}(\alpha)$ tensors. We call $l_{\otimes}(\alpha)$ the length of $\alpha$ with respect to $T_{\otimes}(G)$. The following result is an immediate consequence of \cite[Lemma 2.1]{DS}. 

\begin{lem}\label{lem.coset} Let $C$ be a subgroup of finite index $m$ in $[G,G^{\varphi}]$. Then for every $b \in [G,G^{\varphi}]$,  the coset $Cb$ contains an element $h$ such that $l_{\otimes}(h)\leq m-1$.
\end{lem}

\section{Proofs}

We will now fix some notation and hypothesis.

\begin{hyp}\label{hyp} Let $G$ be a group. Suppose that $C_{\nu(G)}(x)$ has finite index at most $n$ in $\nu(G)$ for each $x\in T_{\otimes}(G)$. Let $m$ be the maximum of indices of $C_{[G,G^{\varphi}]}(x)$ in $[G,G^{\varphi}]$, where $x\in T_{\otimes}(G)$.
Then consider $a\in T_{\otimes}(G)$ such that $C_{[G,G^{\varphi}]}(a)$ has index precisely $m$ in $[G,G^{\varphi}]$. By Lemma ~\ref{lem.coset} we can choose $b_1,\ldots,b_m\in [G,G^{\varphi}]$ such that $l_{\otimes}(b_i)\leq m-1$ and $a^H=\{a^{b_i};\ i=1,\ldots,m\}$. Finally set $U=C_{\nu(G)}(\langle b_1,\ldots,b_m\rangle)$.
\end{hyp}

\begin{rem}
Notice that the subgroup $U$ has finite $n$-bounded index in $\nu(G)$ because $l_{\otimes}(b_i)\leq m-1$ and $C_{\nu(G)}(x)$ has index at most $n$ in $\nu(G)$ for every $x\in T_{\otimes}(G)$. Moreover, from Lemma ~\ref{lem.theta} (b) it follows that the subgroup $\Theta(G)$ is contained in $U$.
\end{rem}

Since Lemma \ref{basic.nu}~(a) shows that the commutator of tensors is still a tensor, from \cite[Lemma 2.3]{DS} we have directly the following.

\begin{lem}\label{lem.comm} Assume Hypothesis ~\ref{hyp}. Then the subgroup $[[G,G^{\varphi}],x]$ has finite $m$-bounded order for any $x\in T_{\otimes}(G)$.
\end{lem}

The next lemma is somewhat analogous to \cite[Lemma 4.5]{W}.

\begin{lem}\label{lem.utheta} Assume Hypothesis~\ref{hyp}. If $u\in U$ and $ua\in~T_{\otimes}(G)$, then $[[G,G^{\varphi}],u]\leq[[G,G^{\varphi}],a]$.
\end{lem}
\begin{proof}  

For every $i=1, \ldots, m$  $(ua)^{b_i}=u a^{b_i}$. Thus the elements $ua^{b_i}$ form the conjugacy class $(ua)^{[G,G^{\varphi}]}$ because $|(ua)^{[G,G^{\varphi}]}| \leq m $. Therefore, whenever we consider an element $h\in [G,G^{\varphi}]$ there exists $i\in\{1,\ldots,m\}$ such that $(ua)^{h}=ua^{b_i}$ and so $u^ha^h=ua^{b_i}$. Hence, $$[u,h]=a^{b_i}a^{-h}=[b_i,a^{-1}][a^{-1},h]\in~[[G,G^{\varphi}],a].$$ The lemma follows.
\end{proof}

\begin{prop}\label{prop.subgroup} Assume Hypothesis~\ref{hyp} and write $a=[d,e^{\varphi}]$ for suitable $d,e\in G$. Then there exists a subgroup $U_1\leq U$ with the following properties.
\begin{enumerate}
\item[1.] The index of $U_1$ in $\nu(G)$ is $n$-bounded;
\item[2.] $[[G,G^{\varphi}],U_1']\leq[[G,G^{\varphi}],a]^{d^{-1}}$;
\item[3.] $[[G,G^{\varphi}],[U_1,d]]\leq[[G,G^{\varphi}],a]$.
\end{enumerate}
\end{prop}
\begin{proof} Set $$U_1=U\cap U^{d^{-1}}\cap U^{d^{-1}e^{-1}}.$$ Since the index of $U$ in $\nu(G)$ is $n$-bounded, we conclude that the index of $U_1$ in $\nu(G)$ is $n$-bounded as well. 

Now, for every $h_1,h_2\in U_1$ we have $$[\rho(h_1)d,e^{\varphi}\rho(h_2)^{\varphi}]=[\rho(h_1),\rho(h_2)^{\varphi}]^d[d,\rho(h_2)^{\varphi}][\rho(h_1),e^{\varphi}]^{dh_2}[d,e^{\varphi}]^{h_2}$$
and so
$$[\rho(h_1)d,e^{\varphi}\rho(h_2)^{\varphi}]^{h_2^{-1}}=[\rho(h_1),\rho(h_2)^{\varphi}]^{dh_2^{-1}}[d,\rho(h_2)^{\varphi}]^{h_2^{-1}}[\rho(h_1),e^{\varphi}]^{d}[d,e^{\varphi}].$$ 

Set $u=[\rho(h_1),\rho(h_2)^{\varphi}]^{dh_2^{-1}}[d,\rho(h_2)^{\varphi}]^{h_2^{-1}}[\rho(h_1),e^{\varphi}]^{d}$. 
Thus, on  the left hand side of the above equation we have a tensor, while the right hand side coincides with $ua$. Now we show that $u \in U$. 

By Lemma~\ref{lem.the}, there exist elements $\alpha, \beta, \gamma \in \Theta(G)$ such that 
\begin{itemize}
\item[(i)] $[\rho(h_1),\rho(h_2)^{\varphi}] = [h_1,h_2] \alpha$; 
\item[(ii)] $[d, \rho(h_2)^{\varphi}]=[d, h_2]\beta$;
\item[(iii)] $[\rho(h_1),e^{\varphi}]=[h_1,e] \gamma$. 
\end{itemize}
From (i) we have $[\rho(h_1),\rho(h_2)^{\varphi}]^{dh_2^{-1}} =([h_1,h_2]\alpha)^{{dh_2^{-1}}}\in U_1^{dh_2^{-1}} \Theta(G)\leq~U.$ 
Point (ii) implies that $[d,\rho(h_2)^{\varphi}]^{h_2^{-1}}=([d, h_2]\beta)^{h_2^{-1}}\in U$. Finally, (iii) shows that $[\rho(h_1),e^{\varphi}]^d=([h_1,e] \gamma)^d\in U_1^dU_1^{ed} \Theta(G) \leq U$. Hence $u \in U $. By Lemma~\ref{lem.utheta}, $[[G,G^{\varphi}],u]\leq[[G,G^{\varphi}],a]$. Since $\Theta(G)$ centralizes $[G,G^{\varphi}]$, it follows that
\begin{equation*}
[[G,G^{\varphi}],[h_1,h_2]^{dh_2^{-1}}[d,h_2]^{h_2^{-1}}[h_1,e]^d]\leq[[G,G^{\varphi}],a],    
\end{equation*}
for any choice of $h_1,h_2\in U_1$. In particular, taking $h_1=1$ we have $[[G,G^{\varphi}],[d,h_2]^{h_2^{-1}}]\leq[[G,G^{\varphi}],a]$, while taking $h_2=1$ it follows that $[[G,G^{\varphi}],[h_1,e]^d] \leq[[G,G^{\varphi}],a]$. Hence we can conclude that $$[[G,G^{\varphi}],[h_1,h_2]^{dh_2^{-1}}]\leq[[G,G^{\varphi}],a].$$ Since $[[G,G^{\varphi}],a]$ is normal in $[G,G^{\varphi}]$, we have $[[G,G^{\varphi}],[h_1,h_2]]\leq~[[G,G^{\varphi}],a]^{d^{-1}}$ and so $[[G,G^{\varphi}],U_1']\leq[[G,G^{\varphi}],a]^{d^{-1}}$, which proves that $U_1$ has property 2. 

Finally, consider again the inclusion $[[G,G^{\varphi}],[d,h_2]^{h_2^{-1}}]\leq[[G,G^{\varphi}],a]$. Since $[[G,G^{\varphi}],a]$ is normal in $[G,G^{\varphi}]$, it follows that $[[G,G^{\varphi}],[U_1,d]]\leq[[G,G^{\varphi}],a]$. Therefore $U_1$ has property 3. as well, and we are done.
\end{proof}

The following result is an immediate consequence of \cite[Theorem 1.1]{DS}.

\begin{prop}\label{prop.pavel}
Let $n$ be a positive integer and $G$ a group. Suppose that the size of the conjugacy class $\left|x^{\nu(G)}\right| \leq n$ for every $x \in T_{\otimes}(G)$. Then the second derived subgroup $G''$ is finite with $n$-bounded order. 
\end{prop}

\begin{proof}
Let $w=[a,b]$ be a commutator of $G$. Since $\rho: \nu(G) \rightarrow G$ is an epimorphism, we deduce that  the commutator $w$ has at most as many conjugate in $G$ as the tensor $[a, b^{\varphi}]$ has in $\nu(G)$. Consequently,  every commutator of $G$ has $n$-boundedly finite conjugacy class in $G$. Therefore \cite[Theorem 1.1]{DS} shows that $|G''|$ is finite and $n$-bounded. 
\end{proof}

We are in a position to give a proof of our main result.

\begin{proof}[Proof of Theorem~A] 
Denote the non-abelian tensor square $[G,G^{\varphi}]$ by $H$. Let $m$ be the maximum of indices of $C_{H}(x)$ in $H$, where $x\in T_{\otimes}(G)$.
Then consider $a\in T_{\otimes}(G)$ such that $C_{H}(a)$ has index precisely $m$ in $H$. By Lemma~\ref{lem.coset} we can choose $b_1,\ldots,b_m\in H$ such that $l_{\otimes}(b_i)\leq m-1$ and $a^H=\{a^{b_i};\ i=1,\ldots,m\}$. Set $U=C_{\nu(G)}(\langle b_1,\ldots,b_m\rangle)$. Note that the index of $U$ in $\nu(G)$ is $n$-bounded. Applying Proposition~\ref{prop.subgroup}, we find a subgroup $U_1$, of $n$-bounded index, such that $[H,U_1']\leq\langle[H,a]^{\nu(G)}\rangle$. Since the index of $U_1$ in $\nu(G)$ is $n$-bounded,  so $H/H\cap U_1$ is. Therefore, we can find $n$-boundedly many tensors $t_1,\ldots,t_s\in T_{\otimes}(G)$ such that $H=\langle t_1,\ldots,t_s,H\cap U_1\rangle$. Let $T$ be the normal closure in $\nu(G)$ of the product of the subgroups $[H,a]$ and $[H,t_i]$ for $i=1,\ldots,s$. By Lemma~\ref{lem.comm} each of these subgroups has $n$-bounded order. Moreover, they have at most $n$ conjugates. Thus, $T$ is a product of $n$-boundedly many finite subgroups, normalizing each other and having $n$-bounded order. We conclude that $T$ has finite $n$-bounded order. Therefore it is sufficient to show that the second derived group of the quotient group $\nu(G)/T$ has finite $n$-bounded order. 

Notice that the derived subgroup of $HU_1$ is contained in $Z(H)$ modulo $T$. Indeed, $HU_1$ is generated by $t_1,\ldots,t_s$ and $U_1$. Thus $[H,t_i] \leq T$ and $[H, U_1'] \leq \langle[H,a]^{\nu(G)}\rangle \leq T$, that is, $t_1,\ldots,t_s\in Z(H)$ and $U_1'\leq Z(H)$ modulo $T$. So we pass to the quotient $\nu(G)/T$ and to avoid complicated notation the images of $\nu(G)$, $H$ and $T_{\otimes}(G)$ will be denoted by the same symbols. 

Let $\mathcal F$ denote the family of subgroups $S\leq \nu(G)$ with the following properties.
\begin{enumerate}
\item $\nu(G)'\leq S$;
\item $S'\leq Z(H)$;
\item $S$ has finite index in $\nu(G)$.
\end{enumerate}

First of all observe that $\mathcal F$ is not empty because $HU_1$ belongs to $\mathcal F$. 
Then let $F \in\mathcal F$ of minimal possible index $j$ in $\nu(G)$. Thus $j$ is $n$-bounded because the index of $U_1$ in $\nu(G)$ is $n$-bounded. Now, we argue by induction on $j$. If $j=1$, then $F=\nu(G)$ and $ \nu(G)' \leq Z(H)$. So $\nu(G)''=1$ and we have nothing to prove. Thus, assume that $j\geq2$.

Let $a_0\in T_{\otimes}(G)$ such that $C_H(a_0)$ has maximal possible index in $H$ and write $a_0=[d,e ^{\varphi}]$ for suitable $d,e\in G$. Firstly assume that both $d$ and $e^{\varphi}$ belong to $F$. Then $a_0 \in F' \leq Z(H)$, and we conclude that $H$ is abelian, that is $H'=1$. Moreover, by Proposition~\ref{prop.pavel} it follows that $G''$ is finite on $n$-bounded order. Therefore \cite[Theorem 3.3]{NR1} implies that $\nu(G)''=[G', (G')^{\varphi}] G'' (G'')^{\varphi}$ is finite. 

Thus, we may assume that at least one among $d$ and $e^{\varphi}$ does not belong to $F$, say $d$. By Proposition~\ref{prop.subgroup} it follows that there exists a subgroup $V$ of $n$-bounded index in $\nu(G)$ such that $[H,[V,d]]\leq[H,a_0]$. Without loss of generality we may assume that $V\leq F$, otherwise we can replace $V$ by $V\cap F$. Let $L=F\langle d\rangle$. Note that $L'=F'[F,d]$. Let $1=g_1,\ldots,g_t$ be a full system of representatives of the right cosets of $V$ in $F$. Then standard commutator identities show that $[F,d]$ is generated by $[V,d]^{g_1},\ldots,[V,d]^{g_t}$ and $[g_1,d],\ldots,[g_t,d]$. Denote by $R$ the normal closure in $\nu(G)$ of the product of the subgroups $[H,{a_0}]^{g_i}$ and $[H, [g_i,d]]$ for $i=1,\ldots,t$.
For every $i=1,\ldots,t$, Lemma ~\ref{lem.the} implies that $[g_i,d]$ is a tensor modulo $\Theta(G)$. Hence, applying Lemma~~\ref{lem.comm}, $[H,{a_0}]^{g_i}$ and $[H, [g_i,d]]$ have finite $n$-bounded order.
Moreover, the hypothesis implies that each of them has at most $n$ conjugates. Thus, $R$ is the product of $n$-boundedly many finite subgroups, normalizing each other and having $n$-bounded orders. We conclude that $R$ has finite $n$-bounded order. Since $F' \leq Z(H)$, $[H,L']=[H,[F,d]]\leq R$. This means that $L' \leq Z(H)$ modulo $R$. Moreover, since $d\not\in F$, the index of $L$ in $\nu(G)$ is strictly smaller than $j$. Therefore, by induction on $j$, the second derived group of $\nu(G)/R$ is finite with $n$-bounded order. Taking into account that also $R$ is finite with $n$-bounded order, we deduce that $\nu(G)''$ is finite with $n$-bounded order. The proof is now complete.
\end{proof}

For the reader's convenience we restate Corollary B. 

\begin{corB}
Let $n$ be a positive integer. Let $G$ be a group in which the derived subgroup $G'$ is finitely generated. Assume that the size of the conjugacy class $|x^{\nu(G)}| \leq n$ for every $x \in T_{\otimes}(G)$. Then $G$ is a BFC-group.
\end{corB}

\begin{proof}
By Neumann's result \cite[Theorem 3.1]{neu}, it is sufficient to prove that the derived subgroup $G'$ is finite. 

By Theorem A, the second derived subgroup $\nu(G)''$ is finite. Since $\nu(G)'' = [G',(G')^{\varphi}] (G'')(G'')^{\varphi}$, it follows that the non-abelian tensor square $[G',(G')^{\varphi}]$ is finite. By Lemma \ref{fg}, the derived subgroup $G'$ is finite. The proof is complete. 
\end{proof}

\subsection{Examples}

The next example shows that Corollary B no longer holds if we get rid of the hypothesis of $G'$ to be finitely generated.

\begin{ex} \label{ex.prufer}
Let $p$ be a prime. We define the semi-direct product $G = A \rtimes C_2$, where $C_2 = \langle d \ | \ d^2=1 \rangle$, $A = C_{p^{\infty}}$ is the Pr\"ufer group and $$a^d =a^{-1},$$ for every $a \in A$.
Then the group $G$ is not a BFC-group, whose commutator subgroup $G'$ is not finitely generated, such that $|x^{\nu(G)}| \leq 4$ for every $x \in T_{\otimes}(G)$.

Since $G'=A$ is a Pr\"ufer group, it follows that $G'$ is not finitely generated, and $G$ is not a BFC-group by Neumann's result \cite[Theorem 3.1]{neu}. Now, for all $g,h \in A$, as $C_2$ is generated by $d$, we have 

\begin{itemize}
    \item $[g,d^{\varphi}]^{h} = [g,(d^{h})^{\varphi}] =[g,(dh^2)^{\varphi}] = [g,(h^2)^{\varphi}][g,d^{\varphi}]^{h^2} =  [g,d^{\varphi}]^{h^2}$;
    \item $[d,g^{\varphi}]^h = [d^{h},g^\varphi] = [dh^2,g^{\varphi}] = [d,g^{\varphi}]^{h^2}[h^2,g^{\varphi}] = [d,g^{\varphi}]^{h^2}$;
    \item $[g,d^{\varphi}]^d = [g^{-1},d^{\varphi}]$ and $[d,g^{\varphi}]^d = [d,(g^{-1})^{\varphi}]$.
\end{itemize}
In particular, $[g,d^{\varphi}]^{h} = [g,d^{\varphi}] $ and $[d,g^{\varphi}]^h = [d,g^{\varphi}]$, for all $g,h \in A$. If $\alpha,\beta\in G$, there exist $a,b \in A$ and $i,j \in \{0,1\}$ such that $\alpha = ad^i$ and $\beta = bd^j$. Therefore,
\begin{eqnarray*}
[\alpha,\beta^{\varphi}] & = & [ad^i, (bd^j)^{\varphi}] \\ 
& = &  ([a,(d^j)^{\varphi}][a,b^{\varphi}]^{d^j})^{d^i}([d^i,(d^j)^{\varphi}][d^i,b^{\varphi}]^{d^j}) \\
 & = &  [a,(d^j)^{\varphi}]^{d^i}[d^i,b^{\varphi}]^{d^j}[d,d^{\varphi}]^{ij} \\ 
 & = &  [a^{\varepsilon_1},(d^j)^{\varphi}][d^i,(b^{\varepsilon_2})^{\varphi}]z,
\end{eqnarray*}
where $z=[d,d^{\varphi}]^{ij} \in  Z(\nu(G))$ and $\varepsilon_k \in \{1,-1\}$, $k=1,2$. Moreover, for every $w \in \nu(G)$ there exist $\gamma \in G$, $c \in A$ and $l \in \{0,1\}$ such that $\rho(w) = \gamma = cd^l$. Now, by Lemma \ref{basic.nu} (d), $[\alpha,\beta^{\varphi}]^{w} = [\alpha,\beta^{\varphi}]^{\rho(w)}$ and 
$$
[\alpha,\beta^{\varphi}]^{\rho(w)} = [\alpha,\beta^{\varphi}]^{\gamma} = ([a^{\varepsilon_1},(d^j)^{\varphi}][d^i,(b^{\varepsilon_2})^{\varphi}]z)^{cd^l} = ([a^{\varepsilon_1},(d^j)^{\varphi}][d^i,(b^{\varepsilon_2})^{\varphi}]z)^{d^l}.  
$$
It follows that, the conjugacy class $$[\alpha,\beta^{\varphi}]^{\nu(G)} \subseteq \{[a^{\varepsilon_1},(d^j)^{\varphi}][d^i,(b^{\varepsilon_2})^{\varphi}]z \ | \  \varepsilon_i \in \{-1,1\} \},$$ and $|[\alpha,\beta^{\varphi}]^{\nu(G)}| \leq 4$, for any $\alpha,\beta\in G$.
\end{ex}

Finally, notice that in Corollary B we cannot replace the hypothesis ``$|x^G| \leq n$ for every  $x \in T_{\otimes}(G)$" by ``$|x^G| \leq n$ for every $x$ commutator of $G$", as the following example shows.

\begin{ex} \label{ex.dihedral}
Let $A=\langle a \rangle $ be an infinite cyclic group and let $C_2=\langle d \ | \ d^2=1 \rangle$ be a cyclic group of order $2$. Then consider the semi-direct product $G = A \rtimes C_2$ where $a^d=a^{-1}$ (infinite dihedral group). Then for every commutator $x$ of $G$ we have $|x^G|\leq 2$. However, the derived subgroup $G'$ is an infinite subgroup of $A$, so $G$ is not a BFC-group.   
\end{ex}

\section{Acknowledgments}
The work of the first author was supported by FAPDF/Brazil and DPI/UnB, while the second author was supported by the ``National Group for Algebraic and Geometric Structures, and their Applications" (GNSAGA - INdAM). Moreover, this study was carried out during the second author's visit to the University of Brasilia. He wishes to thank the Department of Mathematics for the excellent hospitality.

\end{document}